\documentclass{article}  
\usepackage{amsmath} 
\usepackage{latexsym}

\usepackage{url}
\usepackage{amssymb} 
\usepackage{exscale} 
\usepackage{graphicx} 
\usepackage{makeidx}
\usepackage{enumerate}

\usepackage{color}
\usepackage{pstricks} 

\newtheorem{thm}{Theorem}

\newtheorem{con}[thm]{Conjecture}

\newtheorem{lem}[thm]{Lemma}

\date{}

\newenvironment{proof}[1][Proof]
{\par\noindent{\bf #1.} }{\hspace*{\fill}\nolinebreak{$\Box$}\bigskip\par}

\title{\bf On Some Zarankiewicz Numbers and
Bipartite Ramsey Numbers for Quadrilateral}

\author{
\Large Janusz Dybizba\'{n}ski\footnote{{\tiny This research was partially supported by the Polish National Science Centre, contract number DEC-2012/05/N/ST6/03063.}}, Tomasz Dzido\footnote{\tiny This research was partially supported by the Polish National Science Centre grant 2011/02/A/ST6/00201.}\\
\small Institute of Informatics, University of Gda\'{n}sk\small \\[-0.8ex]
\small Wita Stwosza 57, 80-952 Gda\'{n}sk, Poland\\[-0.8ex]
\small {\tt \{jdybiz,tdz\}@inf.ug.edu.pl}\\
\\[-1.5ex]
and\\[-1.5ex]
\\
\Large Stanis\l aw Radziszowski\footnote{\tiny Work done while on sabbatical at the Gda\'{n}sk University of Technology, supported by a grant from
the Polish National Science Centre grant 2011/02/A/ST6/00201.} \\
\small Department of Algorithms and System Modeling \\[-0.8ex]
\small Gda\'{n}sk University of Technology \\[-0.8ex]
\small 80-233 Gda\'{n}sk, Poland \\[-0.8ex]
\small and \\[-0.8ex]
\small Department of Computer Science \\[-0.8ex]
\small Rochester Institute of Technology \\[-0.8ex]
\small Rochester, NY  14623, USA \\[-0.8ex]
\small {\tt spr@cs.rit.edu}
}

\begin{document}

\maketitle
%\footnotetext{* This research was partially supported by the Polish National 
%Science Centre grant 2011/02/A/ST6/00201.}
\thispagestyle{empty}

\vspace{-0.5cm}
\begin{abstract}
The Zarankiewicz number $z(m,n;s,t)$ is the maximum number of edges
in a subgraph of $K_{m,n}$ that does not contain $K_{s,t}$ as a subgraph.
The bipartite Ramsey number $b(n_1, \cdots, n_k)$ is the least positive
integer $b$ such that any coloring of the edges of $K_{b,b}$ with $k$
colors will result in a monochromatic copy of $K_{n_i,n_i}$ in
the $i$-th color, for some $i$, $1 \le i \le k$.
If $n_i=m$ for all $i$, then we denote
this number by $b_k(m)$. In this paper we obtain
the exact values of some Zarankiewicz
numbers for quadrilateral ($s=t=2$), and we derive new bounds
for diagonal multicolor bipartite Ramsey numbers avoiding quadrilateral.
In particular, we prove that $b_4(2)=19$, and establish new general
lower and upper bounds on $b_k(2)$. 
\end{abstract}



\bigskip
\section{Introduction}

%In this paper all graphs are undirected, finite and
%contain neither loops nor multiple edges. Let $G=(V,E)$ be such a graph,
%and let $d_i(v)$ denote the number of the edges in color $i$ incident to $v$.
%The open neighborhood in color $i$ of vertex $v$ is
%$N_i(v) = \{u \in V(G) | \{u,v\} \in E(G)$ and \{$u,v$\} has color $i$\}.   
%
The Zarankiewicz number $z(m,n;s,t)$ is defined as the maximum number
of edges in any subgraph $G$ of the complete bipartite graph $K_{m,n}$,
such that $G$ does not contain $K_{s,t}$ as a subgraph.
Zarankiewicz numbers and related extremal graphs have been
studied by numerous authors, including
K\"{o}v\'{a}ri, S\'{o}s, and Tur\'{a}n \cite{KST},
Reiman \cite{Rei}, Irving \cite{Irv},
and Goddard, Henning, and Oellermann \cite{GHO}.
A compact summary by Bollob\'{a}s can be found in \cite{Bol}.

The bipartite Ramsey number $b(n_1,\cdots , n_k)$ is the least positive
integer $b$ such that any coloring of the edges of the complete
bipartite graph $K_{b,b}$ with $k$ colors will result in a
monochromatic copy of $K_{n_i,n_i}$ in the $i$-th color,
for some $i$, $1 \le i \le k$. If $n_i=m$ for all $i$, then we
will denote this number by $b_k(m)$. The study of bipartite
Ramsey numbers was initiated by Beineke and Schwenk in 1976,
and continued by others, in particular Exoo \cite{Ex1},
Hattingh and Henning \cite{HaH},
Goddard, Henning, and Oellermann \cite{GHO},
and Lazebnik and Mubayi \cite{LaM}.

In the remainder of this paper we consider only
the case of avoiding quadrilateral $C_4$, i.e. the case of $s=t=2$.
Thus, for brevity, in the following the Zarankiewicz numbers
will be written as $z(m,n)$ or $z(n)$, instead of $z(m,n;2,2)$
or $z(n,n;2,2)$, respectively. Similarly, the only type of
Ramsey numbers we will study is the case of $b_k(2)$.

We derive new bounds for $z(m,n)$ and $z(n)$ for some
general cases, and in particular we obtain some exact
values of $z(n)$ for $n=q^2+q-h$ and small $h\ge 0$.
This permits to establish the exact values of $z(n)$ for
all $n\le 21$, leaving the first open case for $n=22$.
We establish new lower and upper bounds on
multicolor bipartite Ramsey
numbers of the form $b_k(2)$, and we compute the exact
value for the first previously open case for $k=4$, namely
$b_4(2)=19$. Now the first open case is for $k=5$,
for which we obtain the bounds $26\le b_5(2)\le 28$.

During the time of reviewing and revising this paper we became aware of some recent independent work by others \cite{DHS, FGGP, SP}  on related problems, which we summarize in Section 5. 

\bigskip
\section{Zarankiewicz Numbers for Quadrilateral}

In 1951, Kazimierz Zarankiewicz \cite{Zar} asked what is the minimum
number of 1's in a 0-1 matrix of order $n \times n$, which
guarantees that it has a $2 \times 2$ minor of 1's. In the notation
introduced above, it asks for the value of $z(n)+1$.

The results and methods used to compute or estimate $z(n)$ are
similar to those in the widely studied case of ${\rm ex} (n,C_4)$,
where one seeks the maximum number of edges in any $C_4$-free
$n$-vertex graph. The latter ones may have triangles
(though not many since no two triangles can share an edge),
which seems to cause that computing ${\rm ex} (n,C_4)$
is harder than $z(m)$, when the number of potential
edges is about the same at $n\approx m\sqrt{2}$.

The main results to date on $z(m,n)$ or $z(n)$ were obtained
in early papers by K\"{o}v\'{a}ri, S\'{o}s, and Tur\'{a}n
(1954, \cite{KST}) and Reiman (1958, \cite{Rei}).
A nice compact summary of what is known
was presented by Bollob\'as \cite{Bol} in 1995.

\begin{thm}
{\bf (\cite{KST}, \cite{Rei}, \cite{Bol})}

\medskip
\noindent
{\rm (a)}
$z(m,n) \leq m/2 +\sqrt{m^2+4mn(n-1)}/2$\ \ 
for all $m, n \ge 1$,

\medskip
\noindent
{\rm (b)}
$z(m) \leq (m +m\sqrt{4m-3})/2$,\ \ 
for all $m \ge 1$,

\medskip
\noindent
{\rm (c)}
$z(p^2+p,p^2)=p^2(p+1)$\ \
for primes $p$,

\medskip
\noindent
{\rm (d)}
$z(q^2+q+1)=(q+1)(q^2+q+1)$\ \ 
for prime powers $q$, and

\medskip
\noindent
{\rm (e)}
$\lim_{n \rightarrow \infty} z(n)/n^{3/2}=1$.
\end{thm}

In Theorem 1, statement (a) with $m=n$ gives statement (b), (c) is an equality in (a) for $m=p^2+p, n=p^2$ and primes $p$, and (d) is an equality in (b)
for $m=q^2+q+1$ for prime powers $q$. Statements (b) and (d) are widely cited
in contrast to somewhat forgotten (a) and (c). The equality in statement (d)
is realized by the point-line bipartite graph of any projective plane
of order $q$. We note that the statement of Theorem 1.3.3. in \cite{Bol}
has a typo in (ii), where instead of $(q-1)$ it should be $(q+1)$.
In the remainder of this section we will derive more cases
similar to statements (c) and (d).  We will be listing explicitly all coefficients in the polynomials involved, hence for easier comparison we restate
(d) as
$$z(k^2+k+1) = k^3 + 2k^2 + 2k +1 \eqno{(1)}$$
for prime powers $k$.
The results for new cases which we will consider include both
lower and upper bounds on $z(n)$ for $n=k^2+k+1-h$ with small
$h$, $1 \le h \le 4$.

\begin{thm}
For prime powers $k$, for $0 \le h \le 4$, and for $n=k^2+k+1-h$,
there exist $C_4$-free subgraphs of $K_{n,n}$ of
sizes establishing lower bounds for $z(n)$ as follows:

$$
z(k^2+k+1-h) \geq
\left\{
\begin{array}{ll}
  k^3+2k^2+2k+1 & \textrm{ for\ \ }h=0,\\
  k^3+2k^2 & \textrm{ for\ \ }h=1,\\
  k^3+2k^2-2k & \textrm{ for\ \ }h=2,\\
  k^3+2k^2-4k+1 & \textrm{ for\ \ }h=3,\ \ and\\
  k^3+2k^2-6k+2 & \textrm{ for\ \ }h=4.\\
\end{array}
\right.
\eqno{(2)}
$$
\label{tabtw}
\end{thm}

\begin{proof}
For each prime power $k$, consider the bipartite graph
$G_k=(P_k\cup B_k,E_k)$ of a projective plane of order $k$,
on the partite sets $P_k$ (points) and $B_k$ (lines).
We have $|P_k|=|B_k|=k^2+k+1$, $|E_k|=k^3+2k^2+2k+1$, and
for $p \in P_k$ and $l \in B_k$,
$\{p,l\}\in E_k$ if and only if point $p$ is on line $l$.
One can easily see that $G_k$ is $(k+1)$-regular and
$C_4$-free. We will construct the induced subgraphs
$H(k,h)$ of $G_k$ by removing $h$ points from $P_k$ and
$h$ lines from $B_k$, where the removed vertices
$\{p_1,\cdots,p_h\} \cup \{l_1,\cdots,l_h\}$ induce
$s(h)$ edges in $G_k$. Then, the number of edges
in $H(k,h)$ is equal to
$$|E_k|-2(k+1)h+s(h).\eqno{(3)}$$

It is easy to choose the removed vertices so that
$s(h)=0,1,3,6,9$ for $h=0,1,2,3,4$, respectively.
The case $h=0$ is trivial,
for $h=1$ we take a point on a line, and for $h=2$
we take points $p_1,p_2$, the line $l_1$ containing
them, and a second line $l_2$ containing $p_2$, so
that $p_1l_1p_2l_2$ forms a path $P_3$.
Consider three points not on a line and three lines
defined by them for $h=3$, then such removed parts induce
a $C_6$. Finally, for $h=4$, we take three collinear
points $\{p_1,p_2,p_3\}$ on line $l_1$, $p_4$ not on $l_1$,
and three lines passing through $p_4$ and the first
three points. It is easy to see that these vertices
induce a subgraph of $K_{4,4}$ with 9 edges.
To complete the proof observe that the right hand sides
of (2) are equal to the values of (3) for corresponding $h$.
\end{proof}

Next, for $1 \le h \le 3$,
we obtain the upper bound on $z(k^2+k+1-h)$
equal to the lower bound in Theorem \ref{tabtw}.
We observe that now we do not require $k$ to be
a prime power, and that obviously the equality holds
in (2) for $h=0$ by (1).

\bigskip
\begin{thm}
For all $k \ge 2$,
$$
z(k^2+k+1-h) \leq
\left\{
\begin{array}{ll}
k^3+2k^2 & \textrm{for\ \ } h=1,\\
k^3+2k^2-2k & \textrm{for\ \ } h=2, \textrm { and }\\
k^3+2k^2-4k+1 & \textrm{for\ \ } h=3.\\
\end{array}
\right.
\eqno{(4)}
$$
\label{tabtw2}
\end{thm}

\begin{proof}
We will proceed with the steps A through G below in a similar way for $h=1,2$ and 3,
and we will label an item by (X.$hi$) if it is a part of step X for $h=i$.
For $h=3$ and $k=2,3$, it is known that $z(4)=9$ and $z(10)=34$ \cite{GHO}
(see also Table 1 below), and these values satisfy (4). Hence, in the
rest of the proof we will assume that $k\ge 4$ for $h=3$.
First we prove that 

\medskip
\begin{enumerate}[({A}.$h$1)]
\item
$z(k^2+1,k^2+k) < (k+1)(k^2+1)$,
\item
$z(k^2-k+1,k^2+k-1) < (k+1)(k^2-k+1)$, and
\item
$z(k^2-2k+2,k^2+k-2) < (k+1)(k^2-2k+2)$.
\end{enumerate}

In (A.$hi$) we aim at the smallest $m$, so that $z(m,n)<(k+1)m$
can still be proven by our method for $n=k^2+k+1-h$.
Suppose for contradiction that a bipartite graph $H$,
with the partite sets $L$ and $R$ of suitable orders,
attains any right hand side in (A). We will count the number
of paths $P_3$ of type $LRL$ in $H$. Since

\medskip
\begin{enumerate}[({B}.$h$1)]
\item
$\frac{(k+1)(k^2+1)}{k^2 + k} = k + \frac{k+1}{k^2+k} = k + \frac{1}{k}$,
\item
$\frac{(k+1)(k^2-k+1)}{k^2 + k-1} = (k-1) + \frac{2k}{k^2+k-1}$, and
\item
$\frac{(k+1)(k^2-2k+2)}{k^2 + k-2} = (k-2) + \frac{4k-2}{k^2+k-2}$,
\end{enumerate}
we conclude that the minimum number of such paths is
achieved in $H$ when $R$ has the degree sequence of

\medskip
\begin{enumerate}[({C}.$h$1)]
\item
$(k+1)$ vertices of degree $(k+1)$ and $(k^2-1)$ vertices of degree $k$,
\item
$2k$ vertices of degree $k$ and $(k^2-k-1)$ vertices of degree $(k-1)$, or
\item
$4k-2$ vertices of degree $(k-1)$ and $(k^2-3k)$ vertices of degree $(k-2)$,
\end{enumerate}
respectively. Hence, the number of $LRL$ paths in $H$ is at least
\begin{enumerate}[({D}.$h$1)]
\item
$(k+1){k+1 \choose 2}+(k^2-1){k \choose 2} = \frac{1}{2}k(k+1)(k^2-k+2)$, 
\item
$2k{k \choose 2}+(k^2-k-1){k-1 \choose 2} = \frac{1}{2}(k-1)(k^3-k^2+k+2)$, or
\item
$(4k-2){k-1 \choose 2}+(k^2-3k){k-2 \choose 2} = \frac{1}{2}(k-2)(k^3-2k^2+3k+2)$.
\end{enumerate}

\medskip
\noindent
On the other hand
\begin{enumerate}[({E}.$h$1)]
\item
${k^2+1 \choose 2} = \frac{1}{2}k^2(k^2+1)$, 
\item
${k^2-k+1 \choose 2} = \frac{1}{2}(k-1)(k^3-k^2+k)$, and
\item
${k^2-2k+2 \choose 2} = \frac{1}{2}(k^4-4k^3+7k^2-6k+2)$.
\end{enumerate}

\medskip
\noindent
Observe that the following hold:
\begin{enumerate}[({F}.$h$1)]
\item
$k(k+1)(k^2-k+2) > k^2(k^2+1)$ for $k \ge 1$,
\item
$(k-1)(k^3-k^2+k+2) > (k-1)(k^3-k^2+k)$ for $k \ge 2$, and
\item
$(k-2)(k^3-2k^2+3k+2) > (k^4-4k^3+7k^2-6k+2)$ for $k \ge 4$,
\end{enumerate}
which imply that (D.$hi$) $>$ (E.$hi$) for the three cases and
for $k$ as specified in (F). Consequently, in all these cases
there exist two $LRL$ paths in $H$ which share both of their endpoints.
This creates $C_4$ which is a contradiction, and thus (A) holds.
Further, we can see that any $C_4$-free bipartite graph with partite
sets $L$ and $R$ of orders as in $H$ must have the minimum degree
on part $L$ at most $k$ (otherwise (A) would not be true).

Finally, consider any $C_4$-free bipartite graph $G$ with
both partite sets of order $n=k^2+k+1-h$. Any of its subgraphs
of partite orders of $H$ must have at least one vertex of degree
at most $k$ in $L$, and together with (A) this implies that $G$
has at most

\medskip
\begin{enumerate}[({G}.$h$1)]
\item
$(k+1)k^2+kk=k^3+2k^2$,
\item
$(k+1)(k^2-k)+k(2k-1)=k^3+2k^2-2k$, or
\item
$(k+1)(k^2-2k+1)+k(3k-3)=k^3+2k^2-4k+1$
\end{enumerate}
edges for $h=1,2,3$, respectively.
These values are the same as the upper bounds claimed
in (4), which completes the proof of Theorem 3. 
\end{proof}

\begin{thm}
For any prime power $k$, and also for $k=1$,
$$
z(k^2+k+1-h) =
\left\{
\begin{array}{ll}
  k^3+2k^2+2k+1 & \textrm{ for\ \ }h=0,\\
  k^3+2k^2 & \textrm{ for\ \ }h=1,\\
  k^3+2k^2-2k & \textrm{ for\ \ }h=2,\  and\\
  k^3+2k^2-4k+1 & \textrm{ for\ \ }h=3.\\
\end{array}
\right.
$$
\label{wnz1}
\label{wnz2}
\end{thm}

\begin{proof}
Theorems 1(d), \ref{tabtw} and \ref{tabtw2} imply the equality
for all prime powers $k$. The easy cases for $k=1$ hold as well,
as can be checked in Table 1.
\end{proof}

\bigskip
Goddard, Henning and Oellermann obtained the value $z(18)=81$,
and their proof is a special case of our Theorems
\ref{tabtw} and \ref{tabtw2} for $k=4$ and $h=3$.
We were not able to prove the general upper bound
of $k^3+2k^2-6k+2$ for $h=4$, but we expect that it is true.
We could only obtain one special case for $k=4$, namely $z(17)=74$,
which is established later in this section in Lemma 6.
Thus, we consider that Theorem 2 and the known values of
$z(n)$ for $n=k^2+k-3$, $k=2,3,4$ (see Table 1),
provide strong evidence for the following conjecture.

\begin{con}
For any prime power $k$,
$$z(k^2+k-3) = k^3+2k^2-6k+2.$$ 
\label{wnz3}
\end{con}

Previous work by others \cite{GHO}, Theorem 4,
computations using {\tt nauty}, the special case
in Lemma 6 below, and the comments above,
give all the values of $z(n)$
for $n \le 21$. They are listed in Table 1, together
with the parameters $k$ and $h$ when applicable.
This leaves $z(22)$ as the first open case.
Note that in Table 1 the only cases not covered by
Theorem 4 or Conjecture 5 are those for $n=8, 14, 15$ and 16.

\bigskip
\begin{center}
\begin{tabular}{|r|cc|c||r|cc|c||r|cc|c|}
\hline
$n$&$k$&$h$&$z(n)$&$n$&$k$&$h$&$z(n)$&$n$&$k$&$h$&$z(n)$\cr
\hline
1&1&2&1&8&&&24&15&&&61\cr
2&1&1&3&9&3&4&29&16&&&67\cr
3&1&0&6&10&3&3&34&17&4&4&74\cr
4&2&3&9&11&3&2&39&18&4&3&81\cr
5&2&2&12&12&3&1&45&19&4&2&88\cr
6&2&1&16&13&3&0&52&20&4&1&96\cr
7&2&0&21&14&&&56&21&4&0&105\cr
\hline
\end{tabular}

\medskip
{\small Table 1: {$z(n)$ for $1 \leq n \leq 21$ with $k$, $h$ for $n=k^2+k+1-h$, $h \le 4$.}}
\end{center}

\bigskip
With the help of the package {\tt nauty}
developed by Brendan McKay \cite{McK},
one can easily obtain the values of $z(n)$ for $n \le 16$
and confirm the values of related numbers and extremal graphs
presented in \cite{GHO}.
However, {\tt nauty} cannot complete this task for $n\ge 17$.
The cases $18 \le n \le 21$ are settled by Theorem 4,
hence we fill in the only missing case of $n=17$ with the following
lemma.

\begin{lem}
$z(17)=74$.
\end{lem}

\begin{proof}
For the upper bound,
suppose that there exists a $C_4$-free bipartite graph
$H = (L \cup R,E)$ with $|L|=|R|=17$, which has 75 edges.
Since $z(16) = 67$, then for every edge $\{u,v\}\in E$
we must have $\deg(u)+\deg(v) \ge 9$. Let $u$ be a vertex
of minimum degree $\delta$ in $L$. Clearly $\delta\le 4$.
Removing $u$ from $L$ gives a subgraph $H'$ of $K_{16,17}$
with $75-\delta$ edges and minimum degree $\delta'$
in the part $R$ of $H'$. Now, $z(16) = 67$ implies that
$\delta+\delta'\ge 8$, which in turn leaves the only
possibility $\delta=\delta'=4$. Hence, all four
neighbors of $u$ in $H$, $\{v_1,v_2,v_3,v_4\}$, must have
degree at least 5. Furthermore, in order to avoid $C_4$,
the neighborhoods $N_i$ of $v_i$ cannot have any other
intersection than $\{u\}$. Thus, $L\setminus \{u\}$
is partitioned into four 4-sets $N_i\setminus \{u\}$
and $\deg(v_i)=5$, for $1\le i \le 4$.
All of the other 13 vertices in $R$ are
not connected to $u$, they can have at most one
adjacent vertex in each of the $N_i$'s, and hence
they have degree at most 4. This implies that $H$
has at most 72 edges, which yields a contradiction.

The lower bound construction with 74 edges is
provided by Theorem 2 with $k=h=4$.
\end{proof}

Finally, we note that the method of the proof of Lemma 6
cannot be applied to the first open case of Conjecture 5,
$z(27)$, since it would require a good bound for an
open case of $z(26)$.

%\eject

\medskip
\section{Bipartite Ramsey Numbers $b_k(2)$}

The determination of values of $b_k(2)$ appears to be difficult.
The only known exact results are:
Beineke and Schwenk proved that $b_2(2)=5$ \cite{BeS}, Exoo
found the second value $b_3(2)=11$ \cite{Ex1}, and in the next
section we show that $b_4(2)=19$.

A construction by Lazebnik and Woldar
\cite{LaW} yields $r_k(C_4) \geq k^2 +2$ for prime powers $k$,
where $r_k(G)$ is the classical Ramsey number defined as
the least $n$ such that there is a monochromatic copy of $G$
in any $k$-coloring of the edges of $K_n$. We use a slight
modification of a similar construction from \cite{LaM}, furthermore
only for the special case of graphs avoiding $C_4$ (versus
$r$-uniform hypergraphs avoiding $K^{(r)}_{2,t+1}$). In addition,
in our case we color the edges of $K_{k^2,k^2}$, while for the
graph case in \cite{LaW,LaM} the edges of $K_{k^2+1}$ are colored.
This gives us a new lower bound on $b_k(2)$, which almost
doubles an easy bound $b_k(2)\ge r_k(C_4)/2$, as follows:

\begin{thm}
For any prime power $k$, we have
$$b_k(2) \geq k^2 +1.$$
\label{bk2}
\end{thm}

\begin{proof}
Let $k$ be any prime power and let $n=k^2$. We will define a
$k$-coloring of the edges of $K_{n,n}$ without monochromatic
$C_4$'s. Let $F$ be a $k$-element field,
and consider the partite sets
$L=\{(a,b) \in F \times F\}$ and
$R=\{(a',b') \in F \times F\}$. Color an edge between
two vertices in $L$ and $R$ with color $\alpha \in F$
if and only if
$$a \cdot a' - b - b' = \alpha.$$

Denote by $G_{\alpha}$ the graph consisting of the edges
in color $\alpha$. We claim that $G_{\alpha}$ contains no
monochromatic copy of $C_4$. First we argue that for $(p_1, s_1)$,
$(p_2,s_2) \in L$, $(p_1, s_1) \neq (p_2,s_2)$, the system
\begin{eqnarray*}
p_1x - s_1 - y & = & \beta\\
p_2x - s_2 - y & = & \beta
\end{eqnarray*}
has at most one solution $(x,y) \in R$ for every $\beta \in F$.
Suppose that:
\setcounter{equation}{4}
\begin{eqnarray}
p_1x - s_1 - y   &=& \beta\\
p_2x - s_2 - y   &=& \beta\\
p_1x' - s_1 - y' &=& \beta\\
p_2x' - s_2 - y' &=& \beta
\end{eqnarray}

\noindent
Adding (6) and (7) and subtracting (5) and (8)
yields $(p_2-p_1)(x-x')=0$, which implies that
$p_1=p_2$ or $x=x'$. If $p_1=p_2$, then (5) and (6)
imply that $s_1=s_2$, yielding a contradiction
$(p_1, s_1) = (p_2,s_2)$. On the other hand, if $x=x'$,
then (6) and (8) imply $y=y'$, which gives $(x, y) = (x',y')$.
\end{proof} 

\bigskip
Our next theorem improves by one the upper bound on $b_k(2)$ established
by Hattingh and Henning in 1998 \cite{HaH}, for all $k \ge 5$.

\begin{thm}
For all $k \ge 5$,
$$b_k(2) \le k^2 + k - 2.$$ 
\label{twbk}
\end{thm}

\begin{proof}
For $n=k^2+k-2$, 
suppose that there exists a $k$-coloring $\mathcal{C}$ of the edges
of $K_{n,n}$ without monochromatic $C_4$'s. Theorem 3 with $h=3$
implies that $\mathcal{C}$ has at most $k^3 + 2k^2 - 4k +1$ edges
in any of the colors, and thus at most $m=k(k^3 + 2k^2 - 4k +1)$
edges in $\mathcal{C}$ are colored. One can easily check that
$m<n^2$ for $k \ge 5$, which completes the proof.
\end{proof}

We note that the bound of Theorem 8 is better than one which
could be obtained by the same method using Theorem 1(b) instead
of Theorem 3.
Observe also that in the proof of Theorem 8 with $k=4$ there is no
contradiction, since using $z(18)=81$ one obtains $n^2=m=324$,
and hence a 4-coloring $\mathcal{C}$ of $K_{18,18}$ is not ruled out.
Indeed, we have constructed a few of them, and one is presented
in the next section.

\section{The Ramsey Number $b_4(2)$}

\begin{thm}
$$b_4(2) = 19.$$ 
\end{thm}

\begin{proof}
The same reasoning as in Theorem 8, but now for $k=4$
and $n=k^2+k-1$, gives $m=4z(19)=352<361=n^2$, which
implies the upper bound. The lower bound follows from
a 4-coloring $\mathcal{D}$ of $K_{18,18}$ without
monochromatic $C_4$'s presented in Figures 1 and 2.
This completes the proof, though we will still give an
additional description and comments on the coloring
$\mathcal{D}$ in the following.
\end{proof}

Goddard et al. \cite{GHO} showed that any extremal graph for $z(18)$
must have the degree sequence $n_4=n_5=9$ on both partite sets.
By using a computer algorithm, we have found that such graph
is unique up to isomorphism, and thus it also must be the same
as one described in the proof of Theorem 2 for $k=4$ and $h=3$.
Let us denote it by $G_{18}$, and consider its labeling as in Figure 1.
Note that four $9\times 9$ quarters of $G_{18}$ have the structure

\[ G_{18} = \left[
\begin{array}{cc}
3C_6 & S^T \\
S    & 9K_2
\end{array}
\right] ,
\]

\medskip

\begin{center}
{\scriptsize
\begin{tabular}{ c | c | c || c | c | c }
\texttt{0 {\bf 1} {\bf 1}} &\texttt{0 0 0} &\texttt{0 0 0} &\texttt{{\bf 1} 0 0} &\texttt{{\bf 1} 0 0} &\texttt{{\bf 1} 0 0} \\
\texttt{{\bf 1} 0 {\bf 1}} &\texttt{0 0 0} &\texttt{0 0 0} &\texttt{0 {\bf 1} 0} &\texttt{0 {\bf 1} 0} &\texttt{0 {\bf 1} 0} \\
\texttt{{\bf 1} {\bf 1} 0} &\texttt{0 0 0} &\texttt{0 0 0} &\texttt{0 0 {\bf 1}} &\texttt{0 0 {\bf 1}} &\texttt{0 0 {\bf 1}} \\
\hline
\texttt{0 0 0} &\texttt{0 {\bf 1} {\bf 1}} &\texttt{0 0 0} &\texttt{{\bf 1} 0 0} &\texttt{0 0 {\bf 1}} &\texttt{0 {\bf 1} 0} \\
\texttt{0 0 0} &\texttt{{\bf 1} 0 {\bf 1}} &\texttt{0 0 0} &\texttt{0 {\bf 1} 0} &\texttt{{\bf 1} 0 0} &\texttt{0 0 {\bf 1}} \\
\texttt{0 0 0} &\texttt{{\bf 1} {\bf 1} 0} &\texttt{0 0 0} &\texttt{0 0 {\bf 1}} &\texttt{0 {\bf 1} 0} &\texttt{{\bf 1} 0 0} \\
\hline
\texttt{0 0 0} &\texttt{0 0 0} &\texttt{0 {\bf 1} {\bf 1}} &\texttt{0 0 {\bf 1}} &\texttt{{\bf 1} 0 0} &\texttt{0 {\bf 1} 0} \\
\texttt{0 0 0} &\texttt{0 0 0} &\texttt{{\bf 1} 0 {\bf 1}} &\texttt{{\bf 1} 0 0} &\texttt{0 {\bf 1} 0} &\texttt{0 0 {\bf 1}} \\
\texttt{0 0 0} &\texttt{0 0 0} &\texttt{{\bf 1} {\bf 1} 0} &\texttt{0 {\bf 1} 0} &\texttt{0 0 {\bf 1}} &\texttt{{\bf 1} 0 0} \\
\hline
\hline
\texttt{{\bf 1} 0 0} &\texttt{{\bf 1} 0 0} &\texttt{0 {\bf 1} 0} &\texttt{{\bf 1} 0 0} &\texttt{0 0 0} &\texttt{0 0 0} \\
\texttt{0 {\bf 1} 0} &\texttt{0 {\bf 1} 0} &\texttt{0 0 {\bf 1}} &\texttt{0 {\bf 1} 0} &\texttt{0 0 0} &\texttt{0 0 0} \\
\texttt{0 0 {\bf 1}} &\texttt{0 0 {\bf 1}} &\texttt{{\bf 1} 0 0} &\texttt{0 0 {\bf 1}} &\texttt{0 0 0} &\texttt{0 0 0} \\
\hline
\texttt{{\bf 1} 0 0} &\texttt{0 {\bf 1} 0} &\texttt{{\bf 1} 0 0} &\texttt{0 0 0} &\texttt{{\bf 1} 0 0} &\texttt{0 0 0} \\
\texttt{0 {\bf 1} 0} &\texttt{0 0 {\bf 1}} &\texttt{0 {\bf 1} 0} &\texttt{0 0 0} &\texttt{0 {\bf 1} 0} &\texttt{0 0 0} \\
\texttt{0 0 {\bf 1}} &\texttt{{\bf 1} 0 0} &\texttt{0 0 {\bf 1}} &\texttt{0 0 0} &\texttt{0 0 {\bf 1}} &\texttt{0 0 0} \\
\hline
\texttt{{\bf 1} 0 0} &\texttt{0 0 {\bf 1}} &\texttt{0 0 {\bf 1}} &\texttt{0 0 0} &\texttt{0 0 0} &\texttt{{\bf 1} 0 0} \\
\texttt{0 {\bf 1} 0} &\texttt{{\bf 1} 0 0} &\texttt{{\bf 1} 0 0} &\texttt{0 0 0} &\texttt{0 0 0} &\texttt{0 {\bf 1} 0} \\
\texttt{0 0 {\bf 1}} &\texttt{0 {\bf 1} 0} &\texttt{0 {\bf 1} 0} &\texttt{0 0 0} &\texttt{0 0 0} &\texttt{0 0 {\bf 1}} \\
\end{tabular}
}

\bigskip
{\small Figure 1: $G_{18}$, the unique extremal graph for $z(18)$.}
\end{center}

\bigskip
\bigskip
\begin{center}
{\scriptsize
\begin{tabular}{ c | c | c || c | c | c }
\texttt{2 {\bf 1} {\bf 1}} &\texttt{3 3 4} &\texttt{4 4 3} &\texttt{{\bf 1} 3 2} &\texttt{{\bf 1} 4 2} &\texttt{{\bf 1} 2 3} \\
\texttt{{\bf 1} 2 {\bf 1}} &\texttt{4 3 3} &\texttt{3 4 4} &\texttt{2 {\bf 1} 3} &\texttt{2 {\bf 1} 4} &\texttt{3 {\bf 1} 2} \\
\texttt{{\bf 1} {\bf 1} 2} &\texttt{3 4 3} &\texttt{4 3 4} &\texttt{3 2 {\bf 1}} &\texttt{4 2 {\bf 1}} &\texttt{2 3 {\bf 1}} \\
\hline
\texttt{4 4 3} &\texttt{2 {\bf 1} {\bf 1}} &\texttt{3 3 4} &\texttt{{\bf 1} 3 2} &\texttt{2 3 {\bf 1}} &\texttt{2 {\bf 1} 4} \\
\texttt{3 4 4} &\texttt{{\bf 1} 2 {\bf 1}} &\texttt{4 3 3} &\texttt{2 {\bf 1} 3} &\texttt{{\bf 1} 2 3} &\texttt{4 2 {\bf 1}} \\
\texttt{4 3 4} &\texttt{{\bf 1} {\bf 1} 2} &\texttt{3 4 3} &\texttt{3 2 {\bf 1}} &\texttt{3 {\bf 1} 2} &\texttt{{\bf 1} 4 2} \\
\hline
\texttt{3 3 4} &\texttt{4 4 3} &\texttt{2 {\bf 1} {\bf 1}} &\texttt{2 4 {\bf 1}} &\texttt{{\bf 1} 3 2} &\texttt{2 {\bf 1} 3} \\
\texttt{4 3 3} &\texttt{3 4 4} &\texttt{{\bf 1} 2 {\bf 1}} &\texttt{{\bf 1} 2 4} &\texttt{2 {\bf 1} 3} &\texttt{3 2 {\bf 1}} \\
\texttt{3 4 3} &\texttt{4 3 4} &\texttt{{\bf 1} {\bf 1} 2} &\texttt{4 {\bf 1} 2} &\texttt{3 2 {\bf 1}} &\texttt{{\bf 1} 3 2} \\
\hline
\hline
\texttt{{\bf 1} 2 4} &\texttt{{\bf 1} 2 4} &\texttt{2 {\bf 1} 3} &\texttt{{\bf 1} 2 2} &\texttt{4 3 4} &\texttt{3 3 4} \\
\texttt{4 {\bf 1} 2} &\texttt{4 {\bf 1} 2} &\texttt{3 2 {\bf 1}} &\texttt{2 {\bf 1} 2} &\texttt{4 4 3} &\texttt{4 3 3} \\
\texttt{2 4 {\bf 1}} &\texttt{2 4 {\bf 1}} &\texttt{{\bf 1} 3 2} &\texttt{2 2 {\bf 1}} &\texttt{3 4 4} &\texttt{3 4 3} \\
\hline
\texttt{{\bf 1} 2 3} &\texttt{2 {\bf 1} 4} &\texttt{{\bf 1} 2 4} &\texttt{3 4 3} &\texttt{{\bf 1} 2 2} &\texttt{4 4 3} \\
\texttt{3 {\bf 1} 2} &\texttt{4 2 {\bf 1}} &\texttt{4 {\bf 1} 2} &\texttt{3 3 4} &\texttt{2 {\bf 1} 2} &\texttt{3 4 4} \\
\texttt{2 3 {\bf 1}} &\texttt{{\bf 1} 4 2} &\texttt{2 4 {\bf 1}} &\texttt{4 3 3} &\texttt{2 2 {\bf 1}} &\texttt{4 3 4} \\
\hline
\texttt{{\bf 1} 4 2} &\texttt{2 3 {\bf 1}} &\texttt{2 4 {\bf 1}} &\texttt{3 4 4} &\texttt{4 3 3} &\texttt{{\bf 1} 2 2} \\
\texttt{2 {\bf 1} 4} &\texttt{{\bf 1} 2 3} &\texttt{{\bf 1} 2 4} &\texttt{4 3 4} &\texttt{3 4 3} &\texttt{2 {\bf 1} 2} \\
\texttt{4 2 {\bf 1}} &\texttt{3 {\bf 1} 2} &\texttt{4 {\bf 1} 2} &\texttt{4 4 3} &\texttt{3 3 4} &\texttt{2 2 {\bf 1}} \\
\end{tabular}
}

\bigskip
{\small Figure 2: $\mathcal{D}$, a 4-coloring of the edges of $K_{18,18}$ without monochromatic $C_4$'s.}
\end{center}

\bigskip

\noindent
where $S$ is the point-block bipartite subgraph of $K_{9,9}$
obtained from the unique Steiner triple system on 9 points with
any three parallel blocks removed (out of the total of 12 blocks).

Coloring $\mathcal{D}$ has 81 edges in each of the four colors,
and each of them induces a graph isomorphic to $G_{18}$. Note
that colors 1 and 2 swap and overlay their corresponding
quarters $3C_6$ and $9K_2$, so that $6K_{3,3}$ is formed.
The colors 3 and 4 have the same structure. We have constructed
8 nonisomorphic colorings with the same properties as those
listed for $\mathcal{D}$, but there may be more of them.
They were constructed as follows: First, we overlayed two quarters
$3C_6$ and $9K_2$ of the two first colors as in $\mathcal{D}$,
and then we applied some heuristics to complete the overlay of
the first two colors. Finally, the bipartite complement of
this overlay was split into colors 3 and 4 by standard
SAT-solvers. These were applied to a naturally constructed
Boolean formula, whose variables decide which of the colors 3 or 4
is used for still uncolored edges, so that no monochromatic
$C_4$ is created. Many successful splits were made, but
only 8 of them were nonisomorphic (20 if
the colors are fixed under isomorphisms), and all
of them have the same structure as $\mathcal{D}$.

\bigskip
The first open case of $b_k(2)$ is now for 5 colors, for
which we know that $26 \le b_5(2) \le 28$.
The lower bound is implied by Theorem 7, while the upper
bound by Theorem 8. We believe that the correct value is 28.

\bigskip
\begin{thm}
$$26 \le b_5(2) \le 28.$$
\end{thm}
\begin{con}
$$b_5(2)=28.$$
\end{con}

\section{Addendum}

We would like to add some notes on other independent work about which we became aware while our paper was in review. This includes the work by Steinbach
and Posthoff~\cite{SP}, who achieved the lower bound construction for
$b_4(2)$, but by very different means. Interestingly, their construction is isomorphic to ours. The essence of our Theorem 3 is subsumed by results in the paper by Dam\'{a}sdi, H\'{e}ger and Sz\H{o}nyi~\cite{DHS}, but our proofs are much simpler. Finally,  
Fenner, Gasarch, Glover and Purewal~\cite{FGGP} wrote a very extensive survey of the area of grid colorings, which are essentially equivalent to edge colorings of complete bipartite graphs.

%\eject


\begin{thebibliography}{99}

\bibitem{BeS}
L. W. Beineke and A. J. Schwenk,
On a Bipartite Form of the Ramsey Problem,
{\it Proceedings of the Fifth British Combinatorial
Conference 1975, Congressus Numerantium},
{\bf XV} (1976) 17--22. 

\bibitem{Bol}
B. Bollob\'{a}s,
Extremal Graph Theory, in
{\it Handbook of Combinatorics}, Vol. II,
Elsevier, Amsterdam 1995, 1231--1292.

\bibitem{DHS}
G. Dam\'{a}sdi, T. H\'{e}ger and T. Sz\H{o}nyi,
The Zarankiewicz Problem, Cages, and Geometries,
{\it Annales Universitatis Scientiarum Budapestinensis de Rolando E\"{o}tv\"{o}s Nominatae, Sectio Mathematica},
{\bf 56} (2013) 3--37.


\bibitem{Ex1}
G. Exoo,
A Bipartite Ramsey Number,
{\it Graphs and Combinatorics},
{\bf 7} (1991) 395--396.

\bibitem{FGGP}
S. Fenner, W. Gasarch, C. Glover and S. Purewal,
Rectangle Free Coloring of Grids,
{\tt http://arxiv.org/pdf/1005.3750.pdf}.


\bibitem{GHO}
W. Goddard, M. A. Henning, and O. R. Oellermann,
Bipartite Ramsey Numbers and Zarankiewicz Numbers,
{\it Discrete Mathematics}, {\bf 219} (2000) 85--95.

\bibitem{HaH}
J. H. Hattingh and M. A. Henning,
Bipartite Ramsey Theory,
{\it Utilitas Mathematica},
{\bf 53} (1998) 217--230.

\bibitem{Irv}
R. W. Irving,
A Bipartite Ramsey Problem and Zarankiewicz Numbers,
{\it Glasgow Mathematical Journal},
{\bf 19} (1978) 13--26.

\bibitem{KST}
T. K\"{o}v\'{a}ri, V. T. S\'{o}s, and P. Tur\'{a}n,
On a problem of K. Zarankiewicz,
{\it Colloquium Mathematicum},
{\bf 3} (1954) 50--57.

\bibitem{LaM}
F. Lazebnik and D. Mubayi,
New Lower Bounds for Ramsey Numbers of Graphs and Hypergraphs,
{\it Advances in Applied Mathematics},
{\bf 28} (2002) 544--559.

\bibitem{LaW}
F. Lazebnik and A. Woldar,
New Lower Bounds on the Multicolor Ramsey Numbers $r_k(4)$,
{\it Journal of Combinatorial Theory, Series B},
{\bf 79} (2000) 172--176.

\bibitem{McK}
B. D. McKay, {\it nauty 2.5},
{\tt http://cs.anu.edu.au/$\sim$bdm/nauty}.

\bibitem{Rei}
I. Reiman,
\"{U}ber ein Problem von K. Zarankiewicz,
{\it Acta Mathematica Academiae Scientiarum Hungaricae},
{\bf 9} (1958) 269--273.

\bibitem{SP}
%B. Steinbach and C. Posthoff,
%Artificial Intelligence and Creativity - Two Requirements to Solve an Extremely %Complex Coloring Problem, ICAART 2013, {\tt http://www.informatik.tu- %freiberg.de/prof2/publikationen/ICAART\_2013\_AICRECP.pdf}
B. Steinbach and Ch. Posthoff,
Extremely Complex 4-Colored Rectangle-Free Grids: Solution of Open Multiple-Valued Problems, {\it Proceedings of the IEEE 42nd International Symposium on Multiple-Valued Logic}, Victoria, British Columbia, Canada, (2012) 37--44.
%{\tt %http://www.informatik.tu-freiberg.de/prof2/publikationen/ISMVL\_2012\_mcfcrfg.pdf
%}


%\bibitem{Rad}
%S. P. Radziszowski,
%Small Ramsey Numbers,
%{\it Electronic Journal of Combinatorics},
%Dynamic Survey {\bf 1}, \#13 (2011), 84 pages,
%{\tt http://www.combinatorics.org}.
%

\bibitem{Zar}
K. Zarankiewicz, Problem P101 (in French),
{\it Colloquium Mathematicum},
{\bf 2} (1951) 301.

\end{thebibliography}
\end{document}